\documentclass[12pt,a4paper,oneside,reqno]{article}

\usepackage[utf8]{inputenc}
\usepackage[english]{babel}
\usepackage[left=1in,right=1in,top=1in,bottom=1in]{geometry}
\usepackage{amsfonts,amssymb,amsmath,amsthm}
\usepackage[symbol]{footmisc}
\usepackage{longtable,array}
\newcolumntype{L}[1]{>{\raggedright\arraybackslash$}p{#1}<{$}}
\usepackage{enumerate}
\usepackage{cite}
\usepackage{xcolor}

\usepackage{hyperref}
\hypersetup{hidelinks}
\usepackage{fancyhdr}
\theoremstyle{plain}
\newtheorem{theorem}{Theorem}

\newtheorem{proposition}[theorem]{Proposition}
\newtheorem{corollary}[theorem]{Corollary}
\theoremstyle{definition}
\newtheorem{remark}[theorem]{Remark}
\newtheorem{definition}[theorem]{Definition}
\newtheorem{example}[theorem]{Example}
\makeatletter
\@removefromreset{theorem}{section}\@removefromreset{equation}{section}

\makeatother
\usepackage{adjustbox}
\usepackage{tikz}\usetikzlibrary{arrows}

\IfFileExists{fouriernc.sty}{\usepackage{fouriernc}}{}

\newcommand{\doi}[1]{\href{https://doi.org/#1}{doi:\detokenize{#1}}}
\newcommand{\arxiv}[1]{\href{https://arxiv.org/abs/#1}{arXiv:#1}}

\usepackage{stmaryrd}

\begin{document}

\noindent\fbox{\parbox{\dimexpr\textwidth-2\fboxsep-2\fboxrule\relax}{\small
\textbf{Corrected version (verified by Claude).} This is a corrected version of the paper R.~Fehlberg J\'unior and I.~Kaygorodov, \emph{On the Kantor product, II}, Carpathian Mathematical Publications, 14 (2022), no.~2, 543--563, \doi{10.15330/cmp.14.2.543-563}.}}

\bigskip

\begin{center}
{\Large On the Kantor product, II}

\bigskip

{\bf Renato Fehlberg J\'unior and Ivan Kaygorodov}
\end{center}

\medskip

\noindent{\small\textbf{Abstract.} We describe the Kantor square (and Kantor product) of multiplications, extending the classification proposed in~\cite{Kay1}. Besides, we explicitly describe the Kantor square of some low dimensional algebras and give constructive methods for obtaining new transposed Poisson algebras and Novikov--Poisson algebras, and for classifying Poisson structures and commutative post-Lie structures on a given algebra.}

\bigskip


\bigskip


\bigskip

\tableofcontents

\bigskip

\section*{Introduction}
\addcontentsline{toc}{section}{Introduction}

The idea of obtaining new objects from old ones by using derivative operations has long been
known in algebra \cite{albert}. In its most general form, the idea was realized by  Malcev \cite{malcev}. Let $M_n$
be an associative algebra of matrices of order $n$ over a field $\mathbb F.$ Assume that some finite collection
$\Lambda=(a_{ij} ,b_{ij} ,c_{ij} )$ of matrices in $M_n$ is given. Denote by $M^{(\Lambda)}_n$ an algebra defined on a space of
matrices in $M_n$ with respect to new multiplication
$x \cdot _{\Lambda} y =
\sum_{i,j} a_{ij}xb_{ij}yc_{ij}.$
It was proved
that every $n$-dimensional algebra over $\mathbb F$ is isomorphic to a subalgebra of $M^{(\Lambda)}_n$ \cite{malcev}.
Other interesting ways to derive the initial multiplication are
isotopes, homotopes and mutations \cite{elduque91, alsaody,    vita, pche2}.
The concept of an isotope was introduced by Albert \cite{albert}.
Let algebras $A$ and $A_0$ have a common
linear space on which right multiplication operators $R_x$ and $R^{(0)}_x$ are defined (for $A$ and $A_0$, resp.).
We say that $A_0$ and $A$ are isotopic if there exist invertible linear operators $\phi, \psi, \xi$ such that
$R^{(0)}_x = \phi R_{x\psi} \xi.$
We call $A_0$ an isotope of $A$.
Let $A$ be an arbitrary associative algebra,
and let $p$, $q$ be two fixed elements of $A.$
Then a
new algebra is derived from $A$ by using the same vector space structure of $A$
but defining a new multiplication
$x * y = x p y - y  q  x.$
The resulting algebra is  called the
$(p, q)$-mutation of the algebra $A$.

The definition of the Kantor product of multiplications comes from the study of certain class of algebras. In 1972, Kantor  introduced the class of conservative algebras \cite{Kan1}, that contains many important classes of algebras (see \cite{klp}), for example, associative, Lie, Jordan and Leibniz algebras. To define what will be called Kantor product, we need to introduce the algebras $U(n)$ (see, for more details, \cite{Kan2,klp}). Consider the space $U(n)$ of all bilinear multiplications on the $n$-dimensional vector space $V_n$. Now, fix a vector $u\in V_n$. For $A,B\in U(n)$ (two multiplications) and $x,y\in V_n$, we set
\begin{equation*}
  x\ast y=\llbracket A,B  \rrbracket  (x,y):=A(u,B(x,y))-B(A(u,x),y)-B(x,A(u,y)).
\end{equation*}

\medskip

This new multiplication is called the (left) Kantor product of the multiplications $A$ and $B$ (it is possible to define the right Kantor product). The Kantor product of a multiplication ``$\cdot$'' by itself will be the Kantor square of ``$\cdot$'':
\begin{equation*}
  x\ast y:=u\cdot(x\cdot y)-(u\cdot x)\cdot y-x\cdot (u\cdot y).
\end{equation*}
It is easy to see that the Kantor square of a multiplication is a particular case of the Malcev construction in non-associative sense.
On the other side,
\begin{enumerate}
    \item in commutative associative case it coincides with a mutation;
\item in left commutative and left symmetric cases it coincides with an isotope.
\end{enumerate}
As in \cite{Kay1}, we will assume that the Kantor product is always the left Kantor product.
In \cite{Kay1}, the Kantor product and the Kantor square of many well-known algebras (for example, associative, (anti)-commutative, Lie, Leibniz, Novikov, dialgebras, and Poisson algebras) were studied, and some results about derivations, automorphisms, ideals, and nilpotent algebras were obtained.

\medskip

In this paper we will continue studying the Kantor product of multiplications for some other algebras. Besides, we compute the Kantor square of some finite dimensional algebras and we relate some cases with the ones studied in \cite{Kay1}. In the last section, we give a constructive method for classifying Poisson structures on a given algebra.
Throughout this manuscript, algebras with product ``$\cdot$'' will have always the product written as $x\cdot y:=xy$.

 \medskip

If we are considering an algebra $A$ with product ``$\bullet$'', we will use the following notation:
\begin{longtable}{rcl}
$As(x,y,z)_{\bullet}$&$=$&$(x\bullet y)\bullet z-x\bullet(y\bullet z),$\\
$J(x,y,z)_\bullet$&$=$&$(x\bullet y)\bullet z+(z\bullet x)\bullet y+(y\bullet z)\bullet x,$\\
$x \circ y$ & $ =$ & $ x \bullet y + y \bullet x.$
\end{longtable}
Also, for an algebra ${\bf A}$ we will write $w \approx v$ if $w-v \in {\rm Ann} \ {\bf A},$
where ${\rm Ann} \ {\bf A}$ is the annihilator of ${\bf A}.$
To avoid some discussion during this manuscript, we consider that all algebras are over a field $\mathbb F$ of characteristic zero. We highlight that many of these results hold in positive characteristic.

\section{Kantor square}

In this section we consider varieties of algebras with a single multiplication and study the Kantor square in these cases.

\subsection{Middle-commutative algebras}

The variety of middle-commutative algebras (or reverse algebras) is defined by the identity $(xy)z=z(yx)$. This nomenclature appeared in \cite{kup1} (see too \cite{Moh1}), in counterpart to the definition of left and right commutative algebras. Note that the variety of middle-commutative algebras contains the variety of (anti)-commutative algebras and it is contained in the variety of flexible algebras.

\begin{proposition}\label{prop1}
  Let $(A,\cdot)$ be a middle-commutative algebra. Then $(A,\ast)$ is a middle-commutative algebra.
\end{proposition}

\begin{proof} We have

\begin{align*}
  (x\ast y)\ast z & =u((u(xy))z)-u(((ux)y)z)-u((x(uy))z) \tag{L1} \\
  & -(u(u(xy)))z+(u((ux)y))z+(u(x(uy)))z \tag{L2}\\
  & -(u(xy))(uz)+((ux)y)(uz)+(x(uy))(uz)\tag{L3}
\end{align*}

\begin{align*}
  z\ast (y\ast x) & =u(z(u(yx)))-u(z((uy)x))-u(z(y(ux))) \tag{l1} \\
  & -(uz)(u(yx))+(uz)((uy)x)+(uz)(y(ux)) \tag{l2}\\
  & -z(u(u(yx)))+z(u((uy)x))+z(u(y(ux))) \tag{l3}
\end{align*}

  Computing $(x\ast y)\ast z- z\ast (y\ast x)$, we can see the following cancelations: $({\rm L1})$ with $({\rm l1})$, $({\rm L2})$ with $({\rm l3})$ and $({\rm L3})$ with $({\rm l2})$. Therefore, $(A,\ast)$ is a middle-commutative algebra.
\end{proof}

\subsection{Weakly associative algebras}
The variety of weakly associative algebras is defined by $As(x,y,z)+As(y,z,x)=As(y,x,z)$ (see, for example, \cite{Rem}).
It contains commutative algebras, associative algebras, Lie algebras and symmetric Leibniz algebras  as subvarieties.
A direct computation shows that:

\begin{proposition}
  Let $(A,\cdot)$ be a weakly associative algebra. Then, $(A,\ast)$ is a commutative algebra if, and only if, $(xu)y=y(ux)$ for all $x,y\in A$. In particular, $(A,\ast)$ is a commutative algebra if $(A,\cdot)$ is middle-commutative. Besides, $(A,\ast)$ is an anticommutative algebra if, and only if, $As(x,y,u)+As(y,x,u)=(xu)y+(yu)x$ for all $x,y\in A$.
\end{proposition}

Besides, note that if $(A,\cdot)$ is a weakly associative algebra, it holds that
\[x\ast y=u(xy)-(ux)y-x(uy)=(xy)u-x(yu)-(xu)y,\]
that is, the weakly associative algebras are the algebras that the left Kantor square is equal to the right Kantor square. A consequence of this is the following theorem, that will give us the analogous results of some algebras considered in \cite{Kay1} (Lemmas $5,6,8$ and $9$):

\begin{proposition}\label{wass}
Let $(A,\cdot)$ be a weakly associative algebra.
\begin{itemize}
    \item[(a)] If $(A,\cdot)$ is a right commutative algebra, then $(A,\ast)$ is a right-commutative algebra.
    \item[(b)] If $(A,\cdot)$ is a right Leibniz algebra, then $\ast=0$.
    \item[(c)] If $(A,\cdot)$ is a right Zinbiel algebra, then $(A,\ast)$ is a right Zinbiel algebra.
    \item[(d)] If $(A,\cdot)$ is a right Novikov algebra, then $(A,\ast)$ is a right Novikov algebra.
\end{itemize}
\end{proposition}

A direct consequence of \cite{Rem}, Proposition $1$, is

\begin{proposition}\label{prop_wassder}
  Let $(A,\cdot)$ be an algebra such that $(A,\ast)$ is a weakly associative algebra. Then, $L^\ast_{x}-R^\ast_{x}$ is a derivation of $(A,\ast)$ for all $x\in A$, where $L^\ast_x$ and $R^\ast_x$ are the operators of left and right multiplication by $x$ in $(A,\ast)$.
\end{proposition}

\subsection{Anti-associative algebra}
The variety of anti-associative algebras (see, for example, \cite{SaTow}) is defined by $(xy)z=-x(yz)$. Note that an anti-associative algebra is a nilpotent algebra of nilpotency index $4$. A direct computation (or using \cite{Kay1}, Lemma $25$) shows that:

\begin{proposition}
  Let $(A,\cdot)$ be an anti-associative algebra. Then $(A,\ast)$ is a nilpotent algebra of nilpotency index at most $3$, and therefore, is anti-associative.
\end{proposition}

\subsection{Quasi-commutative associative algebras}
The variety of quasi-commutative associative algebras is defined by the relations: $(xy)z=z(yx)$ and $As(x,y,z)=0$. Based on a work of Khan (see \cite{Moh1}), these algebras were considered in \cite{kup1} as a generalization of associative-commutative algebras.

\begin{proposition}
  Let $(A,\cdot)$ be a quasi-commutative associative algebra. Then $(A,\ast)$ is an associative-commutative algebra.
\end{proposition}

\begin{proof}
First, by  \cite[Lemma $1$]{Kay1}, we know that $As(x,y,z)_\ast=0$. Besides, note that by associativity and middle-commutativity we obtain
\begin{equation*}
    x\ast y=-x(uy)=-(yu)x=-y(ux)=y\ast x.
\end{equation*}
\end{proof}

\subsection{Quasi-commutative alternative algebras}
The variety of quasi-commutative alternative algebras is defined by the identities
\[(xy)z=z(yx);\,\,\,\,\, (x,y,z)=-(y,x,z)=(y,z,x).\]
The variety of quasi-commutative Jordan algebras is defined by the identities
\[(xy)z=z(yx);\,\,\,\,\, x^2(yx)=(x^2y)x.\]
This class of algebras appeared, for example, in \cite{Moh1}. The variety of quasi-commutative Jordan algebras contains (properly) the variety of Jordan algebras and it is contained properly in the variety of noncommutative Jordan algebras (see \cite{Moh1}).

\begin{proposition}
  Let $(A,\cdot)$ be a quasi-commutative alternative algebra. Then $(A,\ast)$ is a quasi-commutative Jordan algebra.
\end{proposition}

\begin{proof}
First, by Proposition \ref{prop1}, we have $ (x\ast y)\ast z=z\ast (y\ast x)$.
It was already computed the Jordan property for alternative algebras in \cite[Theorem 10]{Kay1}. Since middle-commutativity is stronger than the flexible property, we obtain
\[(x\ast x)\ast(y\ast x)-((x\ast x)\ast y)\ast x= 2(((xuxu)y)(ux)-(xuxu)(y(ux))).\]

By middle-commutativity, $xuxu=uxux$.
Since each $2$-generated alternative algebra is associative, we have
\[(x\ast x)\ast(y\ast x)-((x\ast x)\ast y)\ast x= 2As((xu)^2,y,ux)=2As((ux)^2,y,ux)=0.\]
\end{proof}

\subsection{Mock-Lie algebras}\label{Mok}

The variety of mock-Lie algebras (or Jacobi-Jordan algebras; see \cite{Zus}) is defined by the relations $xy=yx$ and $J(x,y,z)=0$. Note that a mock-Lie algebra is a Jordan algebra.

\begin{proposition}\label{prop_mock}
  Let $(A,\cdot)$ be a mock-Lie algebra. Then $x\ast y=2(xy)u$ and $(A,\ast)$ is a mock-Lie algebra.
\end{proposition}

\begin{proof} Since $J(x,y,u)=0$ and the multiplication is commutative, we have $x\ast y=(xy)u-\big((ux)y+(yu)x\big)=2(xy)u$. Hence $(A,\ast)$ is commutative and $(x\ast y)\ast z= 4(((xy)u)z)u$. Therefore
\begin{equation}\label{ml}
  J(x,y,z)_\ast=4[((xy)u)z+((zx)u)y+((yz)u)x]u.
\end{equation}
Let us show that $((xy)u)z+((zx)u)y+((yz)u)x=0$ in every mock-Lie algebra. The identity $J(xy,u,z)=0$ gives $((xy)u)z=-(uz)(xy)-(z(xy))u$. Summing cyclically over $x,y,z$ and using $z(xy)+x(yz)+y(zx)=J(x,y,z)=0$, we obtain
\begin{equation}\label{ml2}
  ((xy)u)z+((zx)u)y+((yz)u)x=-\big[(xy)(zu)+(yz)(xu)+(zx)(yu)\big].
\end{equation}
For $a,b,c,d\in A$ put $P(a,b,c,d)=(ab)(cd)+(ac)(bd)+(ad)(bc)$. By $J(ab,c,d)=0$ and $J(cd,a,b)=0$, we have $2(ab)(cd)=-((ab)c)d-((ab)d)c-((cd)a)b-((cd)b)a$. Hence $2P(a,b,c,d)=-\sum((pq)r)s$, where the sum is taken over all $12$ choices of an unordered pair $\{p,q\}$ and an ordered pair $(r,s)$ such that $\{p,q,r,s\}=\{a,b,c,d\}$. Grouping the terms with the same $s$, we see that this sum is equal to $\sum_s J(p,q,r)s=0$. Thus $P=0$, the right-hand side of~\eqref{ml2} is equal to $-P(x,y,z,u)=0$, and $J(x,y,z)_\ast=0$ by~\eqref{ml}. Therefore, $(A,\ast)$ is a mock-Lie algebra.
\end{proof}

\subsection{Almost-Lie algebras-1}
The variety of almost-Lie algebras is defined (see \cite{kup1}) by the identities
\[(xy)z=z(yx), \,\,\,\,\, J(x,y,z)=0.\]

Since (anti)-commutativity implies the first property, the variety of almost-Lie algebras contains the variety of Lie and Mock-Lie algebras.

\begin{proposition}\label{prop_almost1}
  Let $(A,\cdot)$ be an almost-Lie algebra. Then $(A,\ast)$ is a mock-Lie algebra (see Section \ref{Mok}).
\end{proposition}

\begin{proof}
We have that
\begin{align*}
  x\ast y  &  =u(xy)-(ux)y-x(uy) =u(xy)+(xy)u=u(xy+yx)= u(x\circ y).
\end{align*}
By the middle-commutativity, $u(xy)=(yx)u$ and $u(yx)=(xy)u$; hence $u(x\circ y)=(x\circ y)u$ and $x\ast y=\frac{1}{2}u\circ(x\circ y)$. The algebra $(A,\circ)$ is a mock-Lie algebra: it is commutative and, by the middle-commutativity, $(x\circ y)\circ z=2((xy)z+(yx)z)$, so that $J(x,y,z)_\circ=2\big(J(x,y,z)+J(y,x,z)\big)=0$. Therefore, $(x\ast y)\ast z=\frac{1}{4}u\circ((u\circ(x\circ y))\circ z)$ and
\[J(x,y,z)_\ast=\frac{1}{4}\,u\circ\big[((x\circ y)\circ u)\circ z+((z\circ x)\circ u)\circ y+((y\circ z)\circ u)\circ x\big]=0\]
by the proof of Proposition~\ref{prop_mock} applied to the mock-Lie algebra $(A,\circ)$.
\end{proof}

\subsection{Almost-Lie algebras-2}
The variety of almost-Lie algebras is defined (see \cite{kz}) by the identities
\[xy=-yx, \,\,\,\,\, J(x,y,z)t=0.\]
The variety of almost-Lie algebras is formed by anticommutative central extensions of Lie algebras.
The most interesting subvariety of the variety of almost-Lie algebras is the variety of anticommutative $\mathfrak{CD}$-algebras,
which is defined by a common property of Lie and Jordan algebras:
{\it every commutator for two right multiplications gives a derivation} \cite{kz}.
For an anticommutative algebra ${\mathfrak L}$ we denote the set of all anticommutative $k$-dimensional central extensions as $\mathfrak{cent}(\mathfrak L).$

\begin{proposition}
 Let $(A,\cdot)$ be an almost-Lie algebra.
Then $(A,\ast)$ is a $2$-step nilpotent anticommutative algebra.
In particular, for each Lie algebra  ${\mathfrak L}$,
we have $(\mathfrak{cent}_k({\mathfrak L}), \ast)  \subseteq \mathfrak{cent}_k({\mathfrak L}, \ast).$
\end{proposition}

\begin{proof}
We have that
\begin{longtable}{l}
$(x\ast y)\ast z    =(u(xy)-(ux)y-x(uy)) \ast z =-J(x,y,u)\ast z =$\\

\multicolumn{1}{r}{$  -(u(J(x,y,u)z)-(uJ(x,y,u))z-J(x,y,u)(uz))=0.$}
\end{longtable}

\end{proof}
\subsection{Two-sided Leibniz algebras}

The variety of two-sided Leibniz algebras \cite{dzhuma19} is defined by the identities
\[(xy)z=z(yx), \,\,\,\,\ J(x,y,z)=0, \,\,\,\, (xy+yx)z=0.\]

Since a two-sided Leibniz algebra is an almost-Lie algebra, we have:

\begin{proposition}
  Let $(A,\cdot)$ be a two-sided Leibniz algebra. Then $\ast=0$.
\end{proposition}

\subsection{CL- and CB-algebras}

Remember that the centralizer of an element $x$ in an algebra $A$ is the set
$C_A(x) = \{y\in A  \ : \  xy = yx = 0\}$.
An algebra $A$ is a CL-algebra if every centralizer in $A$ is an ideal of $A$ (see, for example, \cite{SaTow}). Note that an associative-commutative algebra is a CL-algebra.

\begin{proposition}

\begin{itemize}
  \item[(a)] Let $(A,\cdot)$ be an associative-commutative algebra. Then, $(A,\ast)$ is an associative-commutative algebra and therefore a CL-algebra.
  \item[(b)] Let $(A,\cdot)$ be a CL-algebra. If $y\in C_A(x)$ then $yu\in C_A(x)$. Moreover, if $y\in C_A(ux)$ or $x\in C_A(x)$ then $y\in (C_A(x),\ast).$
\end{itemize}
\end{proposition}

Now, let us consider CB-algebras:
\begin{definition}
  Let $A$ be an algebra. We say that elements $x,y\in A$ (or the pair $(x,y)$) have commutative bonding (CB) if $xy = 0$ implies that
$(xz)y = 0$ for all $z\in A$.
\end{definition}

An algebra $A$ is a CB-algebra if every pair of elements of $A$ have commutative bonding (see, for example, \cite{SaTow}). Note that, for example, right-commutative algebras are CB-algebras. In \cite{SaTow}, it was proven that if $A$ is an anticommutative algebra, then $A$ is a CB-algebra if, and only if, it is an anti-associative algebra. Besides, by \cite[Theorem $3.9$]{SaTow}, if $A$ is an anticommutative algebra, then $A$ is a CL-algebra if, and only if, $A$ is a CB-algebra. Direct from these results, we have:

\begin{proposition}
  Let $(A,\cdot)$ be an anticommutative CB-algebra (CL-algebra). Then $(A,\ast)$ is an anticommutative CB-algebra (CL-algebra).
\end{proposition}

Besides, we have the following result

\begin{proposition}
  \begin{itemize}
  \item[(a)] $(A,\ast)$ is a CB-algebra for any algebra $(A,\cdot)$ such that $(A,\ast)$ is right-commutative. In particular, if $(A,\cdot)$ is a weak-associative and right-commutative algebra, then $(A,\ast)$ is a right-commutative and, therefore, $(A,\ast)$ is a CB-algebra.
  \item[(b)]  Let $(A,\cdot)$ be an associative CB-algebra (CL-algebra). Then $(A,\ast)$ is an associative CB-algebra (CL-algebra).
\end{itemize}

\begin{proof}
Item $(a)$ is a consequence of Proposition \ref{wass}. To prove item $(b)$, first note that associative property implies that if $x\ast y=0$ then $xuy=0$ for all $u\in A$. Therefore, if $x\ast y=0$, we have $(x\ast y)\ast z=xuzuy=0$ for all $u,z\in A$. In the case of CL-algebras, we need to show that $C_A(x)_\ast=\{a\in A;a\ast x=x\ast a=0\}$ is an ideal of $(A,\ast)$. But the proof follows the same arguments.
\end{proof}

\end{proposition}

\subsection{Left-symmetric algebras}\label{lsalgebra}

The variety of left-symmetric algebras (or left pre-Lie algebras) is defined (see \cite{kup1}) by $ As(x,y,z)= As(y,x,z)$. It contains, for example, the variety of right Novikov algebras.

\begin{proposition}
  Let $(A,\cdot)$ be a left-symmetric algebra. Then $(A,\ast)$ is a left-symmetric algebra if, and only if, $As(y,xu,u)\approx As(x,yu,u)$.
\end{proposition}

\begin{proof}
Note that

\begin{align*}
  x\ast y & =u(xy)-(ux)y-x(uy)
           =x(uy)-(xu)y-x(uy)=-(xu)y.
\end{align*}

Then,
\begin{align*}
  As(x,y,z)_\ast & =(x\ast y)\ast z-x\ast(y\ast z)\\
          & =(-(xu)y)\ast z-x\ast(-(yu)z) =(((xu)y)u)z-(xu)((yu)z).\\
\end{align*}

Now, we have
\begingroup\begin{longtable}{lll}
$As(x,y,z)_\ast - As(y,x,z)_\ast$ &$=$&$(((xu)y)u)z-(xu)((yu)z)-(((yu)x)u)z +(yu)((xu)z)$ \\
&$   =$&$(((xu)y-(yu)x)u)z+((yu)(xu)-(xu)(yu))z$ \\
&$   =$&$(((xu)y)u-((yu)x)u+(yu)(xu)-(xu)(yu))z $\\
&$   =$&$(As(xu,y,u)- As(yu,x,u))z=(As(y,xu,u)- As(x,yu,u))z.$
\end{longtable}\endgroup{}

The result follows.

\end{proof}

\subsection{Poisson algebras}

A Poisson algebra $(A,\cdot,\{,\})$ is an algebra with two multiplications such that $(A,\cdot)$ is an associative-commutative algebra, $(A,\{,\})$ is a Lie algebra and they satisfy the following compatibility condition:
\[\{x,yz\}=\{x,y\}z+y\{x,z\}\,\,\,\,\mbox{(Leibniz rule)}.\]

Let $(A,\cdot,\{,\})$ be a Poisson algebra. If we consider the Kantor product of this two multiplications, it was proven in \cite{Kay1} that $\llbracket\{,\},.\rrbracket=0$ and $(A,\llbracket\cdot,\{,\}\rrbracket)$ is a Lie algebra. On the other hand, it is known that if we define the new multiplication $\circ=\cdot +\{,\}$ then $(A,\circ)$ is a noncommutative Jordan algebra. We set the following:

\begin{proposition}
  Let $(A,\circ)$ be as before. Then $(A,\ast)$ is a noncommutative Jordan algebra.
\end{proposition}

\begin{proof}
First of all, let us compute $x\ast y$ for any $x,y\in A$. Using all the properties of $(A,\cdot,\{,\})$ (except commutativity of $\cdot$) we obtain
\begin{longtable}{lll}
 $ x\ast y $& $=$ &
 $u\circ(x\circ y)-(u\circ x)\circ y-x\circ (u\circ y)  =u(xy+\{x,y\})+\{u,xy+\{x,y\}\}$ \\
   && $ -(ux+\{u,x\})y-\{ux+\{u,x\},y\}  -x(uy+\{u,y\})-\{x,uy+\{u,y\}\}$ \\
   & $=$ & $ u\{x,y\}-\{ux,y\}-x(uy)-\{x,uy\} = \{y,u\}x+u\{y,x\}-\{x,u\}y-xuy.$
\end{longtable}

By the last equation, we see that $(A,\ast)$, in general, is neither an (anti)-commutative algebra nor middle-commutative.

Let us now verify that $\ast$ is flexible. Note that
\begin{equation}\label{eq5}
  x\ast y+y\ast x=-xuy-yux=-2xyu.
\end{equation}

Then,
\begingroup\begin{longtable}{l}
$  (x\ast y)\ast x-x\ast(y\ast x)=$ \\
$-x\ast (x\ast y)-2xu(x\ast y) - x\ast(-x\ast y-2xuy)  =-2xu(x\ast y)+2x\ast (xuy)=$ \\

$-2xu(\{y,u\}x+u\{y,x\}-\{x,u\}y-xuy) +2(\{xuy,u\}x+u\{xuy,x\}-\{x,u\}xuy-xuxuy)=$\\
 $-2xu(\{y,u\}x+u\{y,x\})+2(\{xuy,u\}x+u\{xuy,x\})
    =-2(xy\{u,xu\}+uy\{x,xu\})=0.$
\end{longtable}\endgroup{}

Now, let us check that $\ast$ satisfy the Jordan identity. By \eqref{eq5} we have $x\ast x=-x^2u$. Therefore

\begingroup\small\begin{longtable}{l}
$ (x\ast x)\ast(y\ast x)-((x\ast x)\ast y)\ast x  = -(x^2u)\ast(y\ast x)+(x^2u\ast y)\ast x=$\\

$ -\{y\ast x,u\}x^2u-u\{y\ast x,x^2u\} +\{x^2u,u\}y\ast x+x^2u^2(y\ast x)+$\\ \multicolumn{1}{r}{$+(\{y,u\}x^2u+u\{y,x^2u\}-\{x^2u,u\}y-x^2u^2y)\ast x=$}\\

$-2x^2u\{y\ast x,u\}-2xu^2\{y\ast x,x\}  +\{x^2u,u\}y\ast x+x^2u^2(y\ast x)+$\\
\multicolumn{1}{r}{$(2\{y,u\}x^2u+2u^2x\{y,x\}+2xyu\{u,x\}-x^2u^2y)\ast x.$}
\end{longtable}\endgroup{}

Expanding the last expression we obtain the Eq. (${\mathfrak y}$):

\begingroup\begin{longtable}{c}
$ -2x^2u\{\{x,u\}y+u\{x,y\}-\{y,u\}x-xuy,u\} -2xu^2\{\{x,u\}y+u\{x,y\}-\{y,u\}x-xuy,x\}$\\

$+\{x^2u,u\}(\{x,u\}y+u\{x,y\}-\{y,u\}x-xuy) +x^2u^2(\{x,u\}y+u\{x,y\}-\{y,u\}x-xuy)$\\

$ +\{x,u\}(2\{y,u\}x^2u+2u^2x\{y,x\}+2xyu\{u,x\}-x^2u^2y)$\\

$+u\{x,2\{y,u\}x^2u+2u^2x\{y,x\}+2xyu\{u,x\}-x^2u^2y\}$\\

$ -\{2\{y,u\}x^2u+2u^2x\{y,x\}+2xyu\{u,x\}-x^2u^2y,u\}x$\\

$-xu(2\{y,u\}x^2u+2u^2x\{y,x\}+2xyu\{u,x\}-x^2u^2y).$
\end{longtable}\endgroup{}

Now, let us compute separately the terms of the form  \{,\} (and without \{,\}), \{\}\{\} and \{,\{\}\}, since they have no relation in $A$. In each case, we will show that these respective sums must be zero.\\

$\bullet$ Terms of the form \{,\} and without \{,\}.\\

Collecting the respective terms in the last expression we obtain:

\begingroup\small\begin{longtable}{l}
$ -2x^2u\{u,xuy\}-2xu^2\{x,xuy\}+xuy\{u,x^2u\} +x^2u^2(\{x,u\}y+u\{x,y\}-\{y,u\}x-xuy)+$\\

\multicolumn{1}{c}{$ \{x,u\}(-x^2u^2y)-u\{x,x^2u^2y\}-\{u,x^2u^2y\}x-$}\\

\multicolumn{1}{r}{$
xu(2\{y,u\}x^2u+2u^2x\{y,x\}+2xyu\{u,x\}-x^2u^2y)=$}\\

$-2x^2u^2(y\{u,x\}+x\{u,y\})-2x^2u^2(y\{x,u\}+u\{x,y\})+2x^2u^2y\{u,x\} +x^2u^2y\{x,u\}+$\\

\multicolumn{1}{c}{$x^2u^3\{x,y\}-x^3u^2\{y,u\}-x^3u^3y-x^2u^2y\{x,u\}-ux^2(2uy\{x,u\}+u^2\{x,y\})-$}\\

\multicolumn{1}{r}{$u^2x(2xy\{u,x\}+x^2\{u,y\})-2x^3u^2\{y,u\}-2u^3x^2\{y,x\}-2x^2yu^2\{u,x\}+x^3u^3y=0.$}
\end{longtable}\endgroup{}

Now, we consider the remaining terms in (${\mathfrak y}$), that are

\begingroup\small\begin{longtable}{l}
$2x^2u(\{u,\{x,u\}y\}+\{u,u\{x,y\}\}-\{u,\{y,u\}x\})+$\\

\multicolumn{1}{c}{$2xu^2 (\{x,\{x,u\}y\}+\{x,u\{x,y\}\}-\{x,\{y,u\}x\})-$}\\

\multicolumn{1}{r}{$2xu\{u,x\}(\{x,u\}y+u\{x,y\}-\{y,u\}x)+$}\\

$\{x,u\}(2\{y,u\}x^2u+2u^2x\{y,x\}+2xyu\{u,x\})+$\\

\multicolumn{1}{c}{$u(\{x,2\{y,u\}x^2u\}+\{x,2u^2x\{y,x\}\}+\{x,2xyu\{u,x\}\})+
$}\\

\multicolumn{1}{r}{$x(\{u,2\{y,u\}x^2u\}+\{u,2u^2x\{y,x\}\}+\{u,2xyu\{u,x\}\})$}.\\
\end{longtable}\endgroup{}

Since $(A,\{,\})$ is a Lie algebra and using Leibniz rule, we obtain

\begingroup\small\begin{longtable}{l}
$2x^2u(y\{u,\{x,u\}\}+\{x,u\}\{u,y\}+u\{u,\{x,y\}\}-x\{u,\{y,u\}\}-\{y,u\}\{u,x\})+$\\

$2xu^2(y\{x,\{x,u\}\}+\{x,u\}\{x,y\}+u\{x,\{x,y\}\}+\{x,y\}\{x,u\}-x\{x,\{y,u\}\})-$\\

$2xu\{u,x\}(\{x,u\}y+u\{x,y\}-\{y,u\}x)+\{x,u\}(2\{y,u\}x^2u+2u^2x\{y,x\}+2xyu\{u,x\})+$\\

\multicolumn{1}{c}{$2x^2u^2\{x,\{y,u\}\}+2u\{y,u\}\{x,x^2u\}+2u^3x\{x,\{y,x\}\} +2u\{y,x\}\{x,u^2x\}+$}\\

\multicolumn{1}{c}{$2xyu^2\{x,\{u,x\}\}+2u\{u,x\}\{x,xyu\})2x^3u\{u,\{y,u\}\}+2x\{y,u\}\{u,x^2u\}+$}\\

\multicolumn{1}{r}{$2u^2x^2\{u,\{y,x\}\}+2x\{y,x\}\{u,u^2x\}+2x^2yu\{u,\{u,x\}\}+2x\{u,x\}\{u,xyu\}).$}\\
\end{longtable}\endgroup{}
Therefore, we can consider:\\

$\bullet$ Terms of the form \{\}\{\}:\\
\begingroup\small\begin{longtable}{l}
$2x^2u(\{x,u\}\{u,y\}-\{y,u\}\{u,x\})+2xu^2(\{x,u\}\{x,y\}+\{x,y\}\{x,u\})-$\\

$2xu\{u,x\}(\{x,u\}y+u\{x,y\}-\{y,u\}x)+\{x,u\}(2\{y,u\}x^2u+2u^2x\{y,x\}+2xyu\{u,x\})  +$\\

\multicolumn{1}{c}{$2u\{y,u\}\{x,x^2u\}+2u\{y,x\}\{x,u^2x\}+2u\{u,x\}\{x,xyu\}+$}\\

\multicolumn{1}{r}{$2x\{y,u\}\{u,x^2u\}+2x\{y,x\}\{u,u^2x\}+2x\{u,x\}\{u,xyu\}=$}\\

$4xu^2\{x,u\}\{x,y\}-2xu\{u,x\}(\{x,u\}y+u\{x,y\}-\{y,u\}x)+$\\

\multicolumn{1}{c}{$\{x,u\}(2\{y,u\}x^2u+2u^2x\{y,x\}+2xyu\{u,x\})+$}\\

\multicolumn{1}{c}{$2ux^2\{y,u\}\{x,u\}+4u^2x\{y,x\}\{x,u\}+2ux\{u,x\}(y\{x,u\}+u\{x,y\})+$}\\

\multicolumn{1}{r}{$4x^2u\{y,u\}\{u,x\}+2xu^2\{y,x\}\{u,x\}+2xu\{u,x\}(x\{u,y\}+y\{u,x\})=0.$}
\end{longtable}\endgroup{}

$\bullet$ Terms of the form \{,\{\}\}:

\begingroup\small\begin{longtable}{l}
$ 2x^2u(y\{u,\{x,u\}\}+u\{u,\{x,y\}\}-x\{u,\{y,u\}\})+$\\

$2xu^2(y\{x,\{x,u\}\}+u\{x,\{x,y\}\}-x\{x,\{y,u\}\})+2x^2u^2\{x,\{y,u\}\}+2u^3x\{x,\{y,x\}\}+$\\

\multicolumn{1}{r}{$2xyu^2\{x,\{u,x\}\}+2x^3u\{u,\{y,u\}\}+2u^2x^2\{u,\{y,x\}\}+2x^2yu\{u,\{u,x\}\} =0.$}
\end{longtable}\endgroup{}

Therefore, $(A,\ast)$ is a noncommutative Jordan algebra.
\end{proof}

\subsection{On solvable algebras}
Let $(A,\cdot)$ be a solvable algebra of solvability index $s>0$. For example, a metabelian algebra is a solvable algebra of solvability index $2$. We have:

\begin{proposition}
  $(A,\ast)$ is a solvable algebra of solvability index at most $s$.
\end{proposition}

\begin{proof}
  Since $x\ast y=u(xy)-(ux)y-x(uy)$, the result follows directly from an induction argument.
\end{proof}

\subsection{On the nucleus of an algebra}

The nucleus of an algebra $(A,\cdot)$, that we will denote by $N(A,\cdot)$, is the set of elements $n\in (A,\cdot)$ such that $As(n,a,b)=As(a,n,b)=As(a,b,n)=0$, for all $a,b\in (A,\cdot)$. We refer to
  \cite{Boe2}.

\begin{proposition}
  Let $(A,\cdot)$ be an algebra. Then we have:

  \begin{itemize}
    \item [(a)] If $n,u\in N(A,\cdot)$, then $n\in N(A,\ast)$.
    \item [(b)] If $n,nu,un\in N(A,\cdot)$, then $n\in N(A,\ast)$.
  \end{itemize}
\end{proposition}

\begin{proof}
  First note that item $(b)$ follows the same ideas as item $(a)$, but with fewer steps.

  For item $(a)$, since $u\in N(A,\cdot)$, it follows that $x\ast y=-(xu)y$ for all $x,y\in A$. Therefore,

  \begingroup\footnotesize\begin{longtable}{lclcll}
$(n\ast y)\ast z-n\ast (y\ast z)$  &$ =$&$ (((nu)y)u)z-(nu)((yu)z)$ & $=$& $((nu)y)(uz)-(nu)(y(uz))$ & $=$\\

&&$(n(uy))(uz)-(nu)(y(uz))$ & $=$&$ n((uy)(uz))-n(u(y(uz)))$& $= 0.$

  \end{longtable}\endgroup{}
  The other cases are treated in the same way. This ends the proof.
\end{proof}

\subsection{On algebras with involution}

Let $(A,\cdot, \star)$ be an algebra with involution $\star$. Then, we have the following:

\begin{proposition}
  If $u$ is self-adjoint (that is, $u^\star=u$) and is in the center of $(A,\cdot)$ then $\star$ is an involution of $(A,\ast)$.
\end{proposition}

\section{Kantor square of low dimensional  algebras}

In this section, we will study the Kantor square of some low dimensional algebras, computing explicitly them.

\begin{remark}
  Since all algebras $(A,\cdot)$ considered in this section are commutative or anticommutative, we have that $(A,\ast)$ will be also a commutative or anticommutative algebra, respectively. When we describe the multiplication from the basis of an algebra, the zero products will be omitted. Besides, since all computations in this section are standard, all of them will be omitted. The notation of algebras will be the same as in the cited papers.
\end{remark}

\subsection{$3$-dimensional Jordan algebras}
  The Kantor square of an associative-commutative algebra is an associative-commutative (in this case, $x\ast y=-uxy$). In the case of non-associative Jordan algebras we do not have the same behaviour. For this purpose, consider the $3$-dimensional Jordan algebra $\mathbb{T}_{02}^{US}$ given in \cite{GoKaPo} (using the same notation) by

  \begin{center} $e_1^2=e_1, \ e_2^2=e_2, \ e_3^2=e_1+e_2, \  e_1e_3=e_2e_3=\frac{1}{2}e_3.$
\end{center}

Write $u=u_1e_1+u_2e_2+u_3e_3$, where $\{e_1,e_2,e_3\}$ is a basis of $\mathbb{T}_{02}^{US}$. Then, $(\mathbb{T}_{02}^{US},\ast)$ is given by

 \begin{center} $e_1\ast e_1=-u_1e_1, \ e_2\ast e_2=-u_2e_2, \ e_3\ast e_3=-u_2e_1-u_1e_2-u_3e_3,$ \  \\ $e_1\ast e_3=-\frac{u_1}{2}e_3-u_3e_1, \ e_2\ast e_3=-\frac{u_2}{2}e_3-u_3e_2.$
\end{center}

A standard computation shows that $(\mathbb{T}_{02}^{US},\ast)$ is a Jordan algebra if, and only if, $u_1=u_2=0$ or $u_3=0$. In this case, we have:
\begin{enumerate}
    \item If $u_1=u_2=0$ and $u_3\neq 0$, then $(\mathbb{T}_{02}^{US},\ast)\cong (\mathbb{T}_{08}^{AU},\cdot)$;
    \item If $u_1=0, u_2\neq 0$ (or $u_1\neq 0, u_2=0$) and $u_3= 0$ then $(\mathbb{T}_{02}^{US},\ast)\cong (\mathbb{T}_{13},\cdot)$;
    \item If $u_1\neq 0, u_2\neq 0$ and $u_3= 0$, then $(\mathbb{T}_{02}^{US},\ast)\cong (\mathbb{T}_{02}^{US},\cdot)$;
\end{enumerate}

On the other hand, if $u_3\neq 0$ and $u_1,u_2$ are not both zero, $(\mathbb{T}_{02}^{US},\ast)$ will be a non-Jordan commutative algebra,
which does not satisfy almost-Jordan identity
\begin{center}$2((yx)x)x+yx^3=3(yx^2)x,$
\end{center}
generalizing Jordan identity.

The Jordan algebras $\mathbb{T}_{13}$ and $\mathbb{T}_{14}$ from \cite{GoKaPo}, for example, give us Jordan algebras independent from the choice of $u$. In fact, since
\begin{longtable}{ll}
    $(\mathbb{T}_{13},\cdot):$ & $ \,e_1^2=e_1; \,e_1e_2=\frac{1}{2}e_2; \, e_2^2=e_3,$\\
    $(\mathbb{T}_{14},\cdot):$ & $ \,e_1^2=e_1; \,e_1e_2=\frac{1}{2}e_2,$
\end{longtable}
a standard computation shows that
\begin{longtable}{ll}
    $(\mathbb{T}_{13},\ast):$&  $\,e_1\ast e_1=-u_1e_1; \,e_1\ast e_2=-u_1(\frac{1}{2}e_2); \, e_2\ast e_2=-u_1e_3,$\\
    $(\mathbb{T}_{14},\ast):$ & $\,e_1\ast e_1=-u_1e_1; \,e_1\ast e_2=-u_1(\frac{1}{2}e_2).$
\end{longtable}
Therefore, if $u_1\neq 0$, then $(\mathbb{T}_{13},\ast)\cong (\mathbb{T}_{13},\cdot)$ and $(\mathbb{T}_{14},\ast)\cong (\mathbb{T}_{14},\cdot)$.

\subsection{$3$-dimensional anticommutative algebras}
Let us remember that the Kantor square of a Lie algebra is zero.
It is known, that each $3$-dimensional binary-Lie (and Malcev) algebra is Lie.
Hence, we are interested in considering only non-Lie algebras.
By \cite{IsKaVo}, we have the following $3$-dimensional non-Lie anticommutative algebras over $\mathbb{C}$:

\begin{longtable}{ll}
  $A_1^{\alpha}:$ & $ e_1e_2=e_3; \, e_1e_3=e_1+e_3; \, e_2e_3=\alpha e_2;$\\
  $A_2:$& $ e_1e_2=e_1; \, e_2e_3=e_2;$\\
  $A_3:$&$ e_1e_2=e_3; \, e_1e_3=e_1; \, e_2e_3=e_2.$
\end{longtable}

Write $u=u_1e_1+u_2e_2+u_3e_3$. Therefore, we have (using notation from \cite{IsKaVo}, Table $A.1$):

\begin{proposition}\label{prop_3anti}
Let $A$ be a $3$-dimensional non-Lie anticommutative algebra,
then $(A,\ast)$ is a metabelian Lie algebra.
Namely,
\begin{itemize}
  \item[(a)] $(A_1^{\alpha},\ast)$ is   defined by
   \begingroup\small\begin{equation*}
    e_1\ast e_2=u_3((1+\alpha)e_3-\alpha e_2);\,\, e_1\ast e_3=-u_2((1+\alpha)e_3-\alpha e_2);\,\, e_2\ast e_3=u_1((1+\alpha)e_3-\alpha e_2).
  \end{equation*}\endgroup{}

  \item[(b)] $(A_2,\ast)$ is a Lie algebra defined by
  \begin{equation*}
    e_1\ast e_2=u_3e_1;\,\, e_1\ast e_3=-u_2e_1;\,\, e_2\ast e_3=u_1e_1.
  \end{equation*}

 \item[(c)]$(A_3,\ast)$ is a Lie algebra defined by
   \begin{equation*}
    e_1\ast e_2=2u_3e_3;\,\,  e_1\ast e_3=-2u_2e_3;\,\,  e_2\ast e_3=2u_1e_3.
  \end{equation*}

\item[(d)] For each of the algebras $A_1^{\alpha}$, $A_2$, $A_3$ and every nonzero $u\in V_3$, the algebra $(A,\ast)$ is isomorphic either to the nilpotent $3$-dimensional Lie algebra or to the metabelian non-nilpotent $3$-dimensional Lie algebra with $1$-dimensional square, and both cases occur (for $u=0$ the algebra $(A,\ast)$ is zero). In particular, the sets $\{ (A_1^{\alpha},\ast) \}_{u \in V_3}$ $(\alpha\in\mathbb C)$, $\{ (A_2,\ast) \}_{u \in V_3}$, and $\{ (A_3,\ast) \}_{u \in V_3}$ coincide up to isomorphism.

\end{itemize}
\end{proposition}

\begin{proof}
Items $(a)$--$(c)$ are obtained by a direct computation. In each case, the square of $(A,\ast)$ is contained in $\langle v\rangle$, where $v=(1+\alpha)e_3-\alpha e_2$, $v=e_1$, and $v=e_3$, respectively, and for $u\neq0$ it is equal to $\langle v\rangle$. A $3$-dimensional Lie algebra with $1$-dimensional square $\langle v\rangle$ is nilpotent (and then isomorphic to the nilpotent $3$-dimensional Lie algebra) if $v\ast A=0$, and it is isomorphic to the metabelian non-nilpotent $3$-dimensional Lie algebra with $1$-dimensional square otherwise. Now, $v\ast A=0$ if and only if $u_1=0$ and $(1+\alpha)u_2+\alpha u_3=0$ for $A_1^{\alpha}$; $u_2=u_3=0$ for $A_2$; and $u_1=u_2=0$ for $A_3$. This gives $(d)$.
\end{proof}

\subsection{$4$-dimensional binary Lie algebras}

The variety of binary Lie algebras (see \cite{Kuz}) is defined by the relations $xy=-yx$ and $J(xy,x,y)=0$.
Note that every Malcev algebra and anticommutative $\mathfrak{CD}$-algebra is a binary Lie algebra.
From \cite{Kuz}, we have the following non-Lie binary Lie algebras of dimension $4$:

\begin{longtable}{ll}
  $A^0:$ &$e_1e_2=e_3, e_3e_4=e_3;$\\
  $A_{\alpha\neq 2}:$ & $e_1e_2=e_3, e_1e_4=e_1, e_2e_4=e_2, e_3e_4=\alpha e_3;$
\end{longtable}
where $A_{2}$ is a Lie algebra, $A_{-1}$ is a Malcev (non-Lie) algebra,
and $A_{0}$ is an anticommutative (non-Lie) $\mathfrak{CD}$-algebra.
Write $u=u_1e_1+u_2e_2+u_3e_3+u_4e_4$.

\begin{proposition}\label{prop_binlie}
Let $A$ be a $4$-dimensional binary Lie (non-Lie) algebra.
Then $(A,\ast)$ is a nilpotent Lie algebra, which is zero if $u\in\langle e_3\rangle$ and is isomorphic to the algebra with multiplication table $e_1e_2=e_3$ otherwise.

\begin{proof}
It is easy to see that
\begin{enumerate}
\item $(A^0,\ast)$ is a Lie algebra defined by
  \begin{equation*}
    e_1\ast e_2=-u_4e_3;\,\, e_1\ast e_4=u_2e_3;\,\, e_2\ast e_4=-u_1e_3.
  \end{equation*}

 \item  $(A_\alpha,\ast)$ is a Lie algebra defined by
  \begin{equation*}
    e_1\ast e_2=(2-\alpha)u_4e_3;\,\, e_1\ast e_4=-(2-\alpha)u_2e_3;\,\, e_2\ast e_4=(2-\alpha)u_1e_3,
  \end{equation*}
  for $\alpha\neq 0,2$. We have $(A_\alpha,\ast)\cong (A^0,\ast)$.\\
\end{enumerate}
Hence, $(A_\alpha,\ast)$ and $(A^0,\ast)$ are $4$-dimensional anticommutative algebras whose square is contained in $\langle e_3\rangle$ and which are zero if and only if $u_1=u_2=u_4=0$, i.e. $u\in\langle e_3\rangle$.
It is known that every nonzero $4$-dimensional $2$-step nilpotent anticommutative algebra is isomorphic to the algebra with multiplication table $e_1e_2=e_3.$

\end{proof}
\end{proposition}

\section{Kantor product}

\subsection{Generic Poisson structures}\label{ncp}
Let  $(A,\cdot)$ be an algebra, then an anticommutative bilinear mapping $\{,\}$ is called a generic Poisson structure if it satisfies the following compatibility condition:
\[\{x,y \cdot z\}=\{x,y\}\cdot z+y\cdot \{x,z\}\,\,\,\,\mbox{(Leibniz rule)}.\]

Note that Poisson,
non-commutative Poisson, non-associative Poisson, Poisson-Malcev, Malcev-Poisson-Jordan, etc. algebras are particular cases of generic Poisson structures with its underlying algebra.

\begin{proposition}
  Let $\{,\}$ be a generic Poisson structure on $(A,\cdot)$. Then $\llbracket\{,\},\cdot\rrbracket =0$ and $(A,\llbracket\cdot,\{,\}\rrbracket  )$ is an anticommutative algebra.
\end{proposition}

\begin{proof} By the Leibniz rule $\llbracket\{,\},\cdot\rrbracket  =0$ (see \cite{Kay1}, Lemma $5$). On the other hand, $\llbracket\cdot,\{,\}\rrbracket  $ is anticommutative, by a direct computation using Leibniz rule and anti-commutativity of $\{,\}$.

\end{proof}

\subsection{Transposed Poisson algebras}
A transposed
Poisson algebra $(A,\cdot, [,])$ (see \cite{BaBaGuWu}) is an algebra with two bilinear operations such that $(A,\cdot)$ is an associative-commutative algebra and $(A,[,])$ is a
Lie algebra that satisfy the following compatibility condition
\begin{equation}\label{dlr}
2z[x,y]=[zx,y]+[x,zy]\,\,\,\,\mbox{(dual Leibniz rule)}.
\end{equation}

Therefore, we can prove the following.
\begin{proposition} \label{tpoi}
  Let $(A,\cdot, [,])$ be a transposed Poisson algebra. Then $(A,\llbracket \cdot, [,] \rrbracket  )$ is a Lie algebra and $(A,\llbracket [,],\cdot \rrbracket  )$ is a commutative algebra.
\end{proposition}

\begin{proof} First, writing $\ast:=\llbracket\ \cdot, [,]\rrbracket  $ we have $x\ast y= u[x,y]-[u x,y]-[x,u y]=-u[x,y]$. Therefore, $(A,\llbracket \cdot, [,]\rrbracket  )$ is an anticommutative algebra. Since
$(x\ast y)\ast z= u[u[x,y],z]$, we have
$$J(x,y,z)_\ast=u[u[x,y],z]+u[u[z,x],y]+u[u[y,z],x]=0,$$
by \cite[Theorem 2.5. (7)]{BaBaGuWu}.

Note that commutativity of $\llbracket  [,],\cdot \rrbracket  $ is direct from the definition of Kantor product.

\end{proof}

Let us now consider a special example of transposed Poisson algebras constructed in \cite[Theorem 25]{bruno}.
The transposed Poisson algebra
$(\mathcal{W}, \cdot, [,])$ is spanned by generators
$\{L_i, \ I_j \}_{ i,j \in \mathbb{Z}}$.
These generators satisfy
\begin{longtable}{rclrcl}
$[L_m, L_n]$&$=$&$ (m -n)L_{m+n},$ &  $[L_m, I_n]$&$ =$&$ (m-n - a )I_{m+n},$\\
$ L_m \cdot L_n$&$=$&$ w L_{m+n},$ & $L_m \cdot  I_n $&$=$&$ w I_{m+n},$
\end{longtable}
\noindent{}where $w$ is an fixed element from the vector space generated by
$\{L_i, \ I_j \}_{ i,j \in \mathbb{Z}}$ and the multiplication given  by juxtaposition satisfies
$L_iL_j=L_{i+j}$ and $L_iI_j=I_{i+j}.$

\begin{proposition}
 \label{newtrans}
Let $\star=\llbracket [,],\cdot \rrbracket  $ and $\{,\}=\llbracket\ \cdot, [,] \rrbracket  $ be new multiplications defined on multiplications of the transposed Poisson algebra $(\mathcal{W}, \cdot, [,])$  defined above.
Then $(\mathcal{W}, \star, \{, \})$ is a transposed Poisson algebra.
\end{proposition}

\begin{proof}
Let $u =\sum\limits_k (u_k^1 L_k+ u_k^2 I_k)$ be the element to define the Kantor product
and $w =\sum\limits_k (w_k^1 L_k+ w_k^2 I_k)$ be the element to define the multiplication ``$\cdot$'' .
Then
by a straightforward calculation we have only the following nonzero multiplications:
\begingroup\footnotesize\begin{longtable}{rcl}
$L_i \star L_j$&$ =$&$ -\sum\limits_{k,n}\Big((k+n)u_k^1w_n^1 L_{i+j+k+n} + (k+n+a)(u_k^1w_n^2+u_k^2w_n^1)I_{i+j+k+n}\Big)$\\

$L_i \star I_j$&$ =$&$ -\sum\limits_{k,n} (k+n)u_k^1w_n^1 I_{i+j+k+n} $\\

$\{ L_i , L_j \}$&$ =$&$ (j-i)\sum\limits_{k,n}
\Big( u_k^1w_n^1 L_{i+j+k+n} + (u_k^1w_n^2+u_k^2w_n^1)I_{i+j+k+n}\Big)$\\

$\{ L_i , I_j \}$&$ =$&$ (j-i+a)\sum\limits_{k,n}
  u_k^1w_n^1 I_{i+j+k+n}  $\\

\end{longtable}\endgroup{}

Now let us define $\Omega:= i+j+m+k_1+n_1+k_2+n_2.$ Hence,

\begingroup \begin{longtable}{rcl}
$ (L_i \star L_j) \star L_m$&$=$&$
\sum\limits_{k_1,n_1,k_2,n_2}
(k_1+n_1)(k_2+n_2)u_{k_1}^1 w_{n_1}^1 u_{k_2}^1w_{n_2}^1 L_{\Omega}+$\\
&&$\sum\limits_{k_1,n_1,k_2,n_2}
(k_1+n_1)(k_2+n_2+a)u_{k_1}^1 w_{n_1}^1( u_{k_2}^1w_{n_2}^2+u_{k_2}^2w_{n_2}^1) I_{\Omega}+$\\

&&$\sum\limits_{k_1,n_1,k_2,n_2}
(k_1+n_1+a)(k_2+n_2)u_{k_2}^1 w_{n_2}^1( u_{k_1}^1w_{n_1}^2+u_{k_1}^2w_{n_1}^1) I_{\Omega}$\\
&$=$& $L_i \star (L_j \star L_m)$
\end{longtable}\endgroup{}
and
\begin{longtable}{rclcl}
$ (L_i \star L_j) \star I_m$&$=$&$
\sum\limits_{k_1,n_1,k_2,n_2}
(k_1+n_1)(k_2+n_2)u_{k_1}^1 w_{n_1}^1 u_{k_2}^1w_{n_2}^1 I_{\Omega}$&$=$& $L_i \star (L_j \star I_m).$
\end{longtable}
Then $\star$ is associative.

It is easy to see that
\begingroup\footnotesize\begin{longtable}{rclll}
$ L_i \star \{ L_j,  L_m\}$&$=$&$
(j-m)\Big($&$\sum\limits_{k_1,n_1,k_2,n_2}
(k_1+n_1)u_{k_1}^1 w_{n_1}^1 u_{k_2}^1w_{n_2}^1 L_{\Omega}+$\\
&&&
$\sum\limits_{k_1,n_1,k_2,n_2}
(k_1+k_2+n_1+n_2+a)u_{k_1}^1 w_{n_1}^1( u_{k_2}^1w_{n_2}^2+u_{k_2}^2w_{n_2}^1) I_{\Omega}\Big).$\\
\end{longtable}\endgroup{}
On the other hand,
\begingroup\footnotesize\begin{longtable}{rcll}
$ \{L_i \star  L_j,  L_m\}$&$=$&$
\sum\limits_{k_1,n_1,k_2,n_2}
(i+j+k_1+n_1-m)(k_1+n_1)u_{k_1}^1w_{n_1}^1u_{k_2}^1w_{n_2}^1 L_{\Omega}+$\\

&&$\sum\limits_{k_1,n_1,k_2,n_2}
(i+j+k_1+n_1-m)(k_1+n_1)(u_{k_2}^1w_{n_2}^2+u_{k_2}^2w_{n_2}^1)u_{k_1}^1w_{n_1}^1 I_{\Omega}+
$\\

&&$\sum\limits_{k_1,n_1,k_2,n_2}
(i+j+k_1+n_1-m)(k_1+n_1+a)(u_{k_1}^1w_{n_1}^2+u_{k_1}^2w_{n_1}^1)u_{k_2}^1w_{n_2}^1 I_{\Omega}
$

\end{longtable}\endgroup{}
and
\begingroup\footnotesize\begin{longtable}{rcll}
$ \{  L_j,  L_i \star L_m\}$&$=$&$
\sum\limits_{k_1,n_1,k_2,n_2}
(-i-m-k_1-n_1+j)(k_1+n_1)u_{k_1}^1w_{n_1}^1u_{k_2}^1w_{n_2}^1 L_{\Omega}+$\\

&&$\sum\limits_{k_1,n_1,k_2,n_2}
(-i-m-k_1-n_1+j)(k_1+n_1)(u_{k_2}^1w_{n_2}^2+u_{k_2}^2w_{n_2}^1)u_{k_1}^1w_{n_1}^1 I_{\Omega}+
$\\

&&$\sum\limits_{k_1,n_1,k_2,n_2}
(-i-m-k_1-n_1+j)(k_1+n_1+a)(u_{k_1}^1w_{n_1}^2+u_{k_1}^2w_{n_1}^1)u_{k_2}^1w_{n_2}^1 I_{\Omega},$
\end{longtable}\endgroup{}
\noindent{}which gives the dual Leibniz rule (\ref{dlr}) for $\{ L_i, L_j, L_m \}.$
By similar calculation, we have
the dual Leibniz rule (\ref{dlr}) for $\{ L_i, L_j, I_m \},$
which conclude the proof of the statement.
\end{proof}

Another example of a transposed Poisson algebra can be constructed in the following way (see \cite{BaBaGuWu}, Proposition $2.2$): let $(A,\cdot)$ be an associative-commutative algebra and $D$ be a derivation of $A$. Define the multiplication

\begin{center}
    $[x,y]:=xD(y)-D(x)y$,
\end{center}
for all $x,y\in A$. Therefore, $(A,\cdot, [,])$ is a transposed Poisson algebra. Then, we have the following result

\begin{proposition}\label{prop_tpD}
Let $\circ=\llbracket [,],\cdot \rrbracket  $ and $\{,\}=\llbracket\ \cdot, [,] \rrbracket  $ be new multiplications defined on multiplications of the transposed Poisson algebra $(A,\cdot, [,])$ defined above.
Then $(A, \circ, \{, \})$ is a transposed Poisson algebra.
\end{proposition}

\begin{proof}
A direct computation shows that $x\circ y=xyD(u)$ (that is, associative) and $\{x,y\}=-u[x,y]$. By Proposition~\ref{tpoi}, we need only to check the dual Leibniz rule. For this purpose, we have:
\begingroup\begin{longtable}{lll}
$\{z\circ x, y\}+\{x,z\circ y\}$ & $=$ & $-u([zxD(u),y]+[x,zyD(u)])$\\
& $=$ & $-u(zxD(u)D(y)-D(zxD(u))y+xD(zyD(u))-zyD(u)D(x))$\\
& $=$ & $-u(zxD(u)D(y)-zyD(u)D(x)-y(xD(zD(u))+D(x)zD(u)))$\\
&  & $-ux(yD(zD(u))+D(y)zD(u))$\\
& $=$ & $-2uzD(u)(xD(y)-yD(x))=-2uzD(u)[x,y]=2z\circ \{x,y\}$.\\
\end{longtable}\endgroup{}
\end{proof}

\subsection{Pre-Lie Poisson algebras}\label{lplp}
A pre-Lie Poisson algebra  $(A,\cdot, \circ)$ (see \cite{BaBaGuWu}) is an algebra with two bilinear operations such that $(A,\cdot)$ is an associative-commutative algebra, $(A,\circ)$ is a left pre-Lie algebra and the following conditions hold:

\begin{equation*}
  (xy)\circ z=x(y\circ z);\,\,\, (x\circ y)z-(y\circ x)z=x\circ(yz)-y\circ(xz).
\end{equation*}

Here, by a pre-Lie Poisson algebra we mean a right pre-Lie Poisson algebra, using an analogous nomenclature as left and right Novikov-Poisson algebras (see \cite{Kay1}). For the definition of left pre-Lie Poisson algebra, see \cite{Kay1} for the analogous properties (see section of Novikov-Poisson algebras). Note that a Novikov-Poisson algebra is a pre-Lie Poisson algebra.

\begin{proposition}
  Let $(A,\cdot, \circ)$ be a right pre-Lie Poisson algebra. Then $(A,\llbracket \circ,\cdot \rrbracket  )$ is a commutative algebra and $(A,\llbracket \cdot,\circ \rrbracket  )$ is a left pre-Lie algebra.
\end{proposition}

\begin{proof}
$(A,\llbracket \circ,\cdot \rrbracket  )$ is a commutative algebra direct from definition of the product. Now, writing $\ast :=\llbracket \cdot,\circ\rrbracket  $, we have that $x\ast y=-x\circ (uy)$. Therefore

\begin{align*}
  (x\ast y)\ast z-(y\ast x)\ast z & =(x\circ(uy))\circ(uz)-(y\circ(ux))\circ(uz) \\
   & =((x\circ y)u)\circ (uz)-((y\circ x)u)\circ(uz) \\
   & =x\circ(y(u\circ(uz)))-y\circ(x(u\circ(uz)))\\
   & =x\circ((uy)\circ(uz)))-y\circ((ux)\circ(uz)))\\
   & =x\circ(u(y\circ(uz)))-y\circ(u(x\circ(uz)))\\
   & =x\ast(y\ast z)-y\ast(x\ast z),
\end{align*}
and the result follows.

\end{proof}

By the same computations as in the case of left Novikov-Poisson algebras (in \cite{Kay1}), we obtain the following result

\begin{proposition}
  Let $(A,\cdot, \circ)$ be a left pre-Lie Poisson algebra. Then $(A,\llbracket \circ,\cdot \rrbracket  )$ is an associative-commutative algebra and $(A,\llbracket \cdot,\circ \rrbracket  )$ is a right pre-Lie algebra.
\end{proposition}

\subsection{On Novikov-Poisson algebras}\label{subsec_np}
A (left) Novikov-Poisson algebra $(A,\cdot, \circ)$ (see \cite{Kay1}) is an algebra with two bilinear operations such that $(A,.)$ is an associative-commutative algebra, $(A,\circ)$ is a left Novikov algebra and the following conditions hold:

\begin{equation}\label{nva}
  x\circ (yz)=(x\circ y)z
\end{equation}

and

\begin{equation}\label{nvb}
  (xy)\circ z-x(y\circ z)=(xz)\circ y-x(z\circ y).
\end{equation}

In \cite{Kay1} it was proven that $(A,\llbracket \cdot,\circ\rrbracket  )$ is a left Novikov algebra and $(A,\llbracket \circ,\cdot \rrbracket  )$ is an associative-commutative algebra. With this we can obtain the following

\begin{corollary}
  Let $(A,\cdot, \circ)$ be a left Novikov-Poisson algebra. Then $(A,\llbracket \circ,\cdot \rrbracket  ,\llbracket \cdot,\circ\rrbracket  )$ is a left Novikov-Poisson algebra.
\end{corollary}

\begin{proof}
  We need only to prove the remaining properties \eqref{nva} and \eqref{nvb}. For this purpose, write $\star=\llbracket \circ,\cdot\rrbracket  $ and $\bullet=\llbracket \cdot,\circ \rrbracket  $. First, observe that
  $x\star y=-u\circ(xy)$ and $x\bullet y=-(ux)\circ y$. Thus

  \[x\bullet(y\star z)=(ux)\circ (u\circ(yz))=u\circ ((ux)\circ(yz))=u\circ(((ux)\circ y)z)=(x\bullet y)\star z.\]

  To the second property we have

  \begin{multline*}
    (x\star y)\bullet z-x\star(y\bullet z)-((x\star z)\bullet y-x\star(z\bullet y))  = \\
      = (u(u\circ(xy)))\circ z-u\circ(x((uy)\circ z))-(u(u\circ(xz)))\circ y+u\circ(x((uz)\circ y)) \\
      = (u(u\circ(xy)))\circ z-(u(u\circ(xz)))\circ y -u\circ(x((uy)\circ z-(uz)\circ y))\\
      = ((u(u\circ x))y)\circ z-((u(u\circ x))z)\circ y -u\circ(x(u(y\circ z)-u(z\circ y)))\\
      = (u(u\circ x))(y\circ z)-(u(u\circ x))(z\circ y) -u\circ((xu)(y\circ z-z\circ y))=0.\\
  \end{multline*}
\end{proof}

\begin{remark}
  Since we do not use the left-commutativity of $\circ$ in the above computation, we obtain an analogous result in the case of left pre-Lie Poisson algebra (see Section \ref{lplp}).
\end{remark}

In \cite[Theorem 29]{Kay1} (see its corrected version), it was proven that if $(A,\cdot)$ is a unital associative algebra, then $(A,\cdot)$ and $(A,\ast)$ are isomorphic if, and only if, $u$ is invertible.
By the above corollary, we can formulate the following question: are $(A,\cdot, \circ)$ and $(A,\llbracket \circ,\cdot\rrbracket  ,\llbracket\ \cdot,\circ \rrbracket  )$ isomorphic?

In general, this is not true. For this purpose, we will use the classification of $3$-dimensional (right) Novikov-Poisson algebras given in \cite[Theorem 3]{Zak}. If we consider the left Novikov-Poisson algebras given in this theorem (after considering in the Novikov product the opposite product) given by the left Novikov algebras $A_5$, $A_6$ and $A_7$ (notation from \cite{Zak}), all of them results in  $\llbracket \circ,\cdot\rrbracket  =0$ and $\llbracket \cdot,\circ \rrbracket  =0$.

On the other hand, in the next example, $(A,\cdot, \circ)$ and $(A,\llbracket \circ,\cdot \rrbracket  ,\llbracket \cdot,\circ \rrbracket  )$ can be isomorphic. Consider the following algebra (given by $C_8$ in \cite{Zak}, Theorem $3$):

Left Novikov product is given by:
\begin{equation*}
  e_3\circ e_1=e_1;\,\,
  e_3\circ e_2=e_2;\,\,
  e_3\circ e_3=e_3.
\end{equation*}

Associative-commutative product is given by:

\begin{equation*}
  e_1e_3=ae_1;\,\,
  e_2e_3=ae_2;\,\,
  e_3^2=de_1+fe_2+ae_3,
\end{equation*}
where $a\neq0$ (note that for $a\neq 0$ the products $e_1e_3=ae_1+be_2$, $e_2^2=ce_2$ satisfy the identities of left Novikov-Poisson algebras only if $b=c=0$).

In the base $\{e_1,e_2,e_3\}$, write $u=u_1e_1+u_2e_2+u_3e_3$. We will suppose that $u_3\neq 0$. Computing the Kantor product we obtain:

Left Novikov product $\llbracket \cdot,\circ\rrbracket  $:
\begin{equation*}
  \llbracket \cdot,\circ \rrbracket  (e_3,e_1)=-u_3ae_1;\,\,
  \llbracket \cdot,\circ \rrbracket  (e_3,e_2)=-u_3ae_2;\,\,
  \llbracket \cdot,\circ \rrbracket  (e_3,e_3)=-u_3ae_3.
\end{equation*}

Associative-commutative product $\llbracket \circ,\cdot \rrbracket  $:

\begin{equation*}
  \llbracket \circ,\cdot \rrbracket  (e_1,e_3)=-u_3ae_1;\,\,
\llbracket \circ,\cdot \rrbracket  (e_2,e_3)=-u_3ae_2;\,\,
\llbracket \circ,\cdot \rrbracket  (e_3,e_3)=-u_3(de_1+fe_2+ae_3).
\end{equation*}

Thus, $\llbracket \cdot,\circ\rrbracket=-u_3a\circ$ and $\llbracket \circ,\cdot\rrbracket=-u_3\,\cdot$. Consequently, $(A,\cdot, \circ)$ and $(A,\llbracket \circ,\cdot \rrbracket  ,\llbracket \cdot,\circ \rrbracket  )$ are isomorphic if and only if $a=1$. Indeed, if $a=1$, then the map $x\mapsto -u_3^{-1}x$ is an isomorphism. Conversely, let $\phi$ be an isomorphism. Then $\eta=-u_3\phi$ is an automorphism of $(A,\cdot)$ such that $a\eta$ is an automorphism of $(A,\circ)$. Hence $a\eta$ preserves the subspace $\langle e_1,e_2\rangle=\{x \mid x\circ A=0\}$ and maps $e_3$ to the set $e_3+\langle e_1,e_2\rangle$ of left units of $(A,\circ)$. Then the coefficient of $e_3$ in $\eta(e_3^2)$ is equal to $1$, while the coefficient of $e_3$ in $\eta(e_3)^2$ is equal to $a^{-1}$; thus $a=1$.

\section{A method for classifying Poisson structures and commutative post-Lie structures with a given  algebra}

\subsection{A method for classifying Poisson structures with a given  algebra}
Let  $({\mathfrak A}, \cdot)$ be an algebra of dimension $n$.
We say that the multiplication $[,],$ defined on the underlying space of the algebra ${\mathfrak A},$ gives a Poisson structure, if
$({\mathfrak A}, [,])$ is a Lie algebra and multiplications $\cdot$ and $[,]$ are satisfying the Leibniz rule
\begin{center}
    $[x \cdot y, z]=[x,z] \cdot y + x \cdot [y, z].$
\end{center}
Then, for  algebras $({\mathfrak A}, \cdot)$ and $({\mathfrak A}, [,])$  we can associate two elements
${\mathfrak a}$ and ${\mathfrak l}$ from our "big" algebra $U(n)$ (constructed by the Kantor way). It is easy to see that
\begin{center}
     $\llbracket{\mathfrak l} , {\mathfrak a}\rrbracket  =0$ and
     $\llbracket{\mathfrak l} , {\mathfrak l}\rrbracket  =0.$
\end{center}
Moreover, if the basis in $U(n)$ will be chosen by a similar way of  \cite[Section 3]{klp},
 ${\mathfrak l}$ is anti-symmetric on lower indices.
 Hence, we have a constructive method for  description of all Poisson structures on a given algebra. We will illustrate it for a basic example below.

\begin{example}\label{ejemplo}
The Poisson structures defined on the $2$-dimensional Jordan algebra ${\mathfrak A}$ with
multiplication table given by
\begin{center}$e_1\cdot e_1=e_1, \ e_1 \cdot e_2=\frac{1}{2}e_2, \ e_2\cdot  e_1=\frac{1}{2}e_2$
\end{center}
are exactly the multiplications $[e_1,e_2]=-[e_2,e_1]=\gamma e_2$, $\gamma\in\mathbb F$.
\end{example}

\begin{proof}
Let ${\mathfrak l}$ be the element from $U(2),$ which is associated to a Poisson structure defined on ${\mathfrak A}.$
Thanks to \cite{klp}, the multiplication of $U(2)$ is given in terms of ``elementary'' multiplications $\alpha_{i,j}^k$ ($i,j,k=1,2$), such that $\alpha_{i,j}^k(v_t,v_l)=\delta_{it}\delta_{jl} v_k$ for all $t,l =1,2,$ (where $v_k$ are basis vectors of  $2$-dimensional space) and presented below.
\begingroup\footnotesize\setlength{\arraycolsep}{1pt}\begin{eqnarray*}
\begin{array}{l l l l}
\llbracket\alpha_{11}^1 , \alpha_{11}^1 \rrbracket =- \alpha_{11}^1 	 &
\llbracket\alpha_{12}^1 , \alpha_{11}^1 \rrbracket =- \alpha_{12}^1-\alpha_{21}^1 	 &
\llbracket\alpha_{11}^2 , \alpha_{11}^1 \rrbracket = \alpha_{11}^2 	 &
\llbracket\alpha_{12}^2 , \alpha_{11}^1 \rrbracket =0 	    \\

\llbracket\alpha_{11}^1 , \alpha_{12}^1 \rrbracket = 0	&
\llbracket\alpha_{12}^1 , \alpha_{12}^1 \rrbracket = - \alpha_{22}^1 	 &
\llbracket\alpha_{11}^2 , \alpha_{12}^1 \rrbracket = -\alpha_{11}^1+\alpha_{12}^2	&
\llbracket\alpha_{12}^2 , \alpha_{12}^1 \rrbracket = -\alpha_{12}^1 	    \\

\llbracket\alpha_{11}^1 , \alpha_{21}^1 \rrbracket = 0 &
\llbracket\alpha_{12}^1 , \alpha_{21}^1 \rrbracket =- \alpha_{22}^1 	 &
\llbracket\alpha_{11}^2 , \alpha_{21}^1 \rrbracket =-\alpha_{11}^1+\alpha_{21}^2	 &
\llbracket\alpha_{12}^2 , \alpha_{21}^1 \rrbracket =-\alpha_{21}^1  	    \\

\llbracket\alpha_{11}^1 , \alpha_{22}^1 \rrbracket = \alpha_{22}^1 &
\llbracket\alpha_{12}^1 , \alpha_{22}^1 \rrbracket = 0  	 &
\llbracket\alpha_{11}^2 , \alpha_{22}^1 \rrbracket =-\alpha_{12}^1-\alpha_{21}^1+\alpha_{22}^2	 &
\llbracket\alpha_{12}^2 , \alpha_{22}^1 \rrbracket =-2\alpha_{22}^1  	    \\

\llbracket\alpha_{11}^1 , \alpha_{11}^2 \rrbracket = -2\alpha_{11}^2 &
\llbracket\alpha_{12}^1 , \alpha_{11}^2 \rrbracket = \alpha_{11}^1-\alpha_{21}^2-\alpha_{12}^2  	 &
\llbracket\alpha_{11}^2 , \alpha_{11}^2 \rrbracket = 0	 &
\llbracket\alpha_{12}^2 , \alpha_{11}^2 \rrbracket = \alpha_{11}^2  	    \\

\llbracket\alpha_{11}^1 , \alpha_{12}^2 \rrbracket = -\alpha_{12}^2 &
\llbracket\alpha_{12}^1 , \alpha_{12}^2 \rrbracket =  \alpha_{12}^1-\alpha_{22}^2  	 &
\llbracket\alpha_{11}^2 , \alpha_{12}^2 \rrbracket = -\alpha_{11}^2	 &
\llbracket\alpha_{12}^2 , \alpha_{12}^2 \rrbracket = 0  	    \\

\llbracket\alpha_{11}^1 , \alpha_{21}^2 \rrbracket = -\alpha_{21}^2 &
\llbracket\alpha_{12}^1 , \alpha_{21}^2 \rrbracket =  \alpha_{21}^1-\alpha_{22}^2  	 &
\llbracket\alpha_{11}^2 , \alpha_{21}^2 \rrbracket = -\alpha_{11}^2	 &
\llbracket\alpha_{12}^2 , \alpha_{21}^2 \rrbracket = 0  	    \\

\llbracket\alpha_{11}^1 , \alpha_{22}^2 \rrbracket = 0  &
\llbracket\alpha_{12}^1 , \alpha_{22}^2 \rrbracket =  \alpha_{22}^1   	 &
\llbracket\alpha_{11}^2 , \alpha_{22}^2 \rrbracket = -\alpha_{12}^2-\alpha_{21}^2	 &
\llbracket\alpha_{12}^2 , \alpha_{22}^2 \rrbracket = -\alpha_{22}^2
\end{array}
\end{eqnarray*}\endgroup{}
\medskip
Hence,  ${\mathfrak a}= \alpha_{11}^1+\frac{1}{2}(\alpha_{12}^2+\alpha_{21}^2)$
and  ${\mathfrak l}=
\gamma_1(\alpha_{12}^1-\alpha_{21}^1)+
\gamma_2(\alpha_{12}^2-\alpha_{21}^2).$
From $\llbracket {\mathfrak l} , {\mathfrak a} \rrbracket =0$, we conclude that
$\gamma_1=0$. Conversely, a direct verification shows that for every $\gamma_2\in\mathbb F$ the multiplication $[e_1,e_2]=-[e_2,e_1]=\gamma_2 e_2$ satisfies the Leibniz rule, and it is a Lie multiplication, since every $2$-dimensional anticommutative algebra is a Lie algebra.

\end{proof}

\subsection{A method for classifying commutative post-Lie structures on a given Lie  algebra}
A commutative post-Lie structure on a Lie algebra $(A,[,])$ is a bilinear product $x\cdot y$ that satisfy (see \cite{BuEn}):

\begin{center}
$x\cdot y=y \cdot x$, $[x,y]\cdot z=x \cdot (y\cdot z)-y\cdot (x\cdot z)$ and $x\cdot [y,z]=[x\cdot y,z]+[y,x\cdot z].$
\end{center}

A direct consequence is the following proposition.

\begin{proposition}
  Let $(A,\cdot)$ be a commutative post-Lie structure on a Lie algebra $(A,[,])$. Then:
  \begin{itemize}
     \item [(b)] $\llbracket\ \cdot,[,] \rrbracket  =0$;
     \item [(c)] $(A,\llbracket\ [,],\cdot\rrbracket  )$ is a commutative algebra.
  \end{itemize}
\end{proposition}

Let  $({\mathfrak A}, [ ,])$ be a Lie algebra of dimension $n$ with
 a given  commutative post-Lie structure $({\mathfrak A}, \cdot ).$
Then, for  algebras $({\mathfrak A}, \cdot)$ and $({\mathfrak A}, [,])$  we can associate two elements
${\mathfrak a}$ and ${\mathfrak l}$ from our "big" algebra $U(n)$ (constructed by the Kantor way). It is easy to see that
\begin{center}
     $\llbracket {\mathfrak a} , {\mathfrak l}\rrbracket   = 0.$
\end{center}
Moreover, if the basis in $U(n)$ will be chosen by a similar way of  \cite[Section 3]{klp},
 ${\mathfrak a}$ is symmetric on lower indices.
 If there are some non-zero elements ${\mathfrak a}$ satisfying the relation indicated above,
 we choose only such commutative multiplications which satisfies the second post-Lie identity  $[x,y] \cdot z=x \cdot (y\cdot z)-y\cdot (x\cdot z).$
 Hence, we have a constructive method for  description of all commutative post-Lie  structures on a given Lie algebra. The present method can be inverted to find commutative post-Lie algebra structures for a given commutative algebra.
 We will illustrate it for a basic example:

\begin{example}\label{ex_postlie}
Let ${\mathfrak S}_2$  be the solvable $2$-dimensional Lie algebra   with
the multiplication table given by $[e_1,e_2]=e_2, \ [e_2,e_1]=- e_2.$
Then, a nonzero commutative post-Lie structure on ${\mathfrak S}_2$ is given (after a change of the basis of ${\mathfrak S}_2$ by an automorphism of ${\mathfrak S}_2$) by one of the following commutative multiplications
\begin{center}
    (I) $e_1 \cdot e_1 =e_2$; \ \
    (II) $e_1 \cdot e_2 =e_2$.
\end{center}
\end{example}

\begin{proof}
Let ${\mathfrak l}$ be the element from $U(2),$ which is associated to a commutative post-Lie structure defined on ${\mathfrak S}_2.$
Thanks to \cite{klp}, the multiplication of $U(2)$ is given in the Example \ref{ejemplo}.
Hence,
\begin{center}${\mathfrak l}= \alpha_{12}^2-\alpha_{21}^2$
and  ${\mathfrak a}=
\sum\limits_{j=1}^2 (\gamma_1^j \alpha_{11}^j+\gamma_2^j(\alpha_{12}^j+\alpha_{21}^j)+
\gamma_3^j \alpha_{22}^j).$\end{center}
From $\llbracket {\mathfrak a} , {\mathfrak l}\rrbracket=0$, we conclude that
\begin{center} ${\mathfrak a}= \gamma_3^1 \alpha_{22}^1 + \gamma_1^2 \alpha_{11}^2+ \gamma_2^2 (\alpha_{21}^2+\alpha_{12}^2)+\gamma_3^2 \alpha_{22}^2.$
\end{center}
Hence,
\begin{center}
    $e_1 \cdot e_1 = \gamma_1^2 e_2, \ e_1 \cdot e_2 = \gamma_2^2 e_2, \ e_2 \cdot e_2 = \gamma_3^1 e_1+\gamma_3^2 e_2.$
\end{center}
It follows that ``$\cdot$'', for each $z=z_1 e_1+z_2e_2$, satisfies only the following additional relation:
\begin{center}
    $e_2 \cdot z = e_1 \cdot (e_2 \cdot z) - e_2 \cdot (e_1 \cdot z).$
\end{center}
Hence,
\begin{longtable}{rcl}
$z_2\gamma_3^1$&$=$&$-z_1 \gamma_1^2 \gamma_3^1 - z_2 \gamma_2^2 \gamma_3^1,$\\
$z_1\gamma_2^2+z_2 \gamma_3^2$&$=$&$z_1(\gamma_2^2)^2+z_2 \gamma_1^2 \gamma_3^1 - z_1 \gamma_1^2 \gamma_3^2,$
\end{longtable}
\noindent{}which gives us that $\gamma_1^2$ is an arbitrary element, $\gamma_3^1=\gamma_3^2=0,$ and $\gamma_2^2=0$ or $\gamma_2^2=1.$
Then we have only two types of commutative post-Lie structures defined on ${\mathfrak S}_2:$

\begin{center}
    (1) $e_1 \cdot e_1 = \gamma_1^2 e_2$; \ \
    (2) $e_1 \cdot e_1 = \gamma_1^2 e_2, \, e_1 \cdot e_2 =e_2.$
\end{center}
In the case (1) with $\gamma_1^2\neq 0$, the change of basis $e_1^*=e_1, \, e_2^*=\gamma_1^2 e_2$ gives (I). In the case (2), the change of basis $e_1^*=e_1-\frac{\gamma_1^2}{2}e_2, \, e_2^*=e_2$ gives (II). Both changes of basis are given by automorphisms of ${\mathfrak S}_2$, and we have the statement of our example.
\end{proof}

\appendix
\section*{Appendix A. Differences from the published version}
\addcontentsline{toc}{section}{Appendix A. Differences from the published version}

\noindent \textbf{Subsection 1.2} (Pseudo-flexible algebras) has been deleted.

\noindent\textbf{Propositions~\ref{prop_mock} and ~\ref{prop_almost1}:} the statement has been changed ($(A,\ast)$ is always a mock-Lie algebra) and the proof completed.

\noindent\textbf{Propositions~\ref{prop_3anti}(d) and~\ref{prop_binlie}:} the statement has been changed and the proof corrected.

\noindent\textbf{Examples~\ref{ejemplo} and~\ref{ex_postlie}:} the statement has been changed  and the proof corrected.

\end{document}